\documentclass[a4paper,english,numberwithinsect,11pt]{lipicsplain-v2016}

\usepackage{amstext}

\usepackage{lineno,hyperref}
\usepackage{amssymb,amsfonts,amsmath}
\modulolinenumbers[5]
\usepackage{mathtools}
\usepackage{paralist}

\usepackage[usenames,dvipsnames,svgnames,table,x11names]{xcolor}
\usepackage{color,soul}
\soulregister\cite7
\soulregister\ref7
\soulregister\pageref7

\usepackage{pgf,tikz}
\usepackage{mathrsfs}
\usetikzlibrary{arrows,calc}
\definecolor{xdxdff}{rgb}{0.49019607843137253,0.49019607843137253,1.}
\definecolor{qqqqcc}{rgb}{0.,0.,0.8}

\sethlcolor{SeaGreen1}
\usepackage{todonotes}

\usepackage{xspace}

\usepackage[title,titletoc,toc]{appendix}
\usepackage[capitalize]{cleveref}

\crefname{appsec}{Appendix}{Appendices}

\theoremstyle{plain}
\newtheorem{thm}[theorem]{Theorem}
\newtheorem{cor}[theorem]{Corollary}
\newtheorem{prop}[theorem]{Proposition}
\newtheorem{lem}[theorem]{Lemma}
\theoremstyle{definition}
\newtheorem{rmk}[theorem]{Remark}
\newtheorem{defn}[theorem]{Definition}
\Crefname{thm}{Theorem}{Theorems}
\Crefname{lem}{Lemma}{Lemmas}

\newenvironment{pf}{\begin{proof}}{\end{proof}}
\usepackage{nicefrac}
\usepackage{xspace}

\newcommand{\erk}{\textnormal{er}(\KCC)}
\newcommand{\erkg}{\textnormal{er}(\KCC(\GCC))}
\newcommand{\opa}{\textnormal{OPT}_\textnormal{A}}
\newcommand{\opb}{\textnormal{OPT}_\textnormal{B}}
\newcommand{\opt}[1]{\textnormal{OPT}_{#1}}

\newcommand{\OPTA}{\opt\threeOMAS(\GCC)}
\newcommand{\OPTB}{\opt\MaxMM(\KG)}
\newcommand{\ALGA}{m_\threeOMAS (\GCC,\ACC(\GCC,\VSS))}
\newcommand{\ALGB}{m_\MaxMM (\KG,\VSS)}

\newcommand{\KG}{\KCC(\GCC)}
\newcommand{\KGH}{\KCC(\GCC,\HCC)}
\newcommand{\tildeKG}{\widetilde{\KCC}(\GCC)}
\newcommand{\edgecardinality}{\left|\ECC\right|}

\newcommand{\inDG}[1]{\deg^{-}_{\GCC}({#1})}
\newcommand{\outDG}[1]{\deg^{+}_{\GCC}({#1})}
\newcommand{\inDH}[1]{\deg^{-}_{\HCC}({#1})}
\newcommand{\outDH}[1]{\deg^{+}_{\HCC}({#1})}
\newcommand{\vertH}{\VCC_\HCC}
\newcommand{\edgeH}{\ECC_\HCC}

\DeclareMathOperator{\conv}{conv}

\newcommand{\dcomplex}[2]{{#1}^{(#2)}}
\newcommand{\twocomplex}[1]{\dcomplex{#1}{2}}

\newcommand{\ACC}{A}
\newcommand{\BCC}{B}
\newcommand{\CCC}{C}
\newcommand{\DCC}{\mathcal{D}}
\newcommand{\ECC}{E}
\newcommand{\FCC}{F}
\newcommand{\GCC}{G}
\newcommand{\HCC}{H}

\newcommand{\KCC}{\mathcal{K}}
\newcommand{\LCC}{\mathcal{L}}
\newcommand{\MCC}{\mathcal{M}}

\newcommand{\QCC}{\mathcal{Q}}

\newcommand{\SCC}{\mathcal{S}}

\newcommand{\VCC}{V}

\newcommand{\HKCC}{\HCC_\KCC}

\newcommand{\VSS}{\mathcal{V}}
\newcommand{\WSS}{\mathcal{W}}

\newcommand{\MaxMM}{\textsf{MaxMM}\xspace}
\newcommand{\MinMM}{\textsf{MinMM}\xspace}
\newcommand{\MAS}{\textsf{MAS}\xspace}
\newcommand{\minFAS}{\textsf{minFAS}\xspace}
\newcommand{\threeMAS}{\textsf{3MAS}\xspace}

\newcommand{\OMAS}{\textsf{OMAS}\xspace}

\newcommand{\threeOMAS}{\textsf{3OMAS}\xspace}

\DeclareMathOperator{\im}{im}

\usepackage{babel}
\begin{document}

\title{Hardness of Approximation for Morse Matching\footnote{Research supported by the DFG Collaborative Research Center TRR 109 \emph{Discretization in Geometry and Dynamics}.}}
\author{Ulrich Bauer}
\author{Abhishek Rathod}

\affil{Department of Mathematics,
Technical University of Munich, Germany. \texttt{\texttt{mail@ulrich-bauer.org, rathod@ma.tum.de}}}
\authorrunning{U. Bauer and A. Rathod} %

\Copyright{Ulrich Bauer and A. Rathod}%

\EventEditors{}
\EventNoEds{2}
\EventLongTitle{34th Symposium on Computational Geometry (SoCG 2018)}
\EventShortTitle{SoCG 2018}
\EventAcronym{SoCG}
\EventYear{2018}
\EventDate{}
\EventLocation{}
\EventLogo{}
\SeriesVolume{}
\ArticleNo{}
\DOIPrefix{}

\maketitle

\begin{abstract}

Discrete Morse theory has emerged as a powerful tool for a wide range of problems, including the computation of (persistent) homology. In this context, discrete Morse theory is used to reduce the problem of computing a topological invariant of an input simplicial complex to computing the same topological invariant of a (significantly smaller) collapsed cell or chain complex. Consequently, devising methods for obtaining gradient vector fields on complexes to reduce the size of the problem instance has become an emerging theme over the last decade. While computing the optimal gradient vector field on a simplicial complex is NP-hard, several heuristics have been observed to compute near-optimal gradient vector fields on a wide variety of datasets. Understanding the theoretical limits of these strategies is therefore a fundamental problem in computational topology. 

In this paper, we consider the approximability of maximization and minimization variants of the Morse matching problem. We establish hardness results for  \emph{Max-Morse matching} and  \emph{Min-Morse matching}, settling an open problem posed by Joswig and Pfetsch~\cite{JP06}. In particular, we show that, for a simplicial complex of dimension $d \geq 3$ with $n$ simplices, it is NP-hard to approximate \emph{Min-Morse matching} within a factor of $O(n^{1-\epsilon})$, for any $\epsilon>0$. Moreover, we establish hardness of approximation results for \emph{Max-Morse matching} for simplicial complexes of dimension $d \geq 2$, using an L-reduction from \emph{Degree 3 Max-Acyclic Subgraph} to \emph{Max-Morse matching}.

\end{abstract}

\section{Introduction\label{sec:intro}}

Classical Morse theory~\cite{Milnor} provides a method to analyze the topology of a smooth manifold by studying the critical points of smooth functions defined on it. Forman's discrete Morse theory is a combinatorial analogue of Morse theory that is applicable to regular cell complexes~\cite{Fo98}. It has become a popular tool in computational topology and visualization~\cite{Caz03,DeyWangWang,shiva2012}, and is actively studied in algebraic, geometric, and topological combinatorics~\cite{Jollenbeck2009,Ko08,Mil07}. In Forman's theory, discrete Morse functions play the role of smooth Morse functions, whereas discrete gradient vector fields are the analogues of gradient-like vector fields. Forman's theory also has an equivalent graph theoretic formulation~\cite{Ch00}, in which the acyclic matchings (or \emph{Morse matchings}) in the Hasse diagram of a simplicial complex correspond to the discrete gradient vector fields on the simplicial complex.  We shall use the terms \emph{gradient vector fields} and \emph{Morse matchings} interchangeably. In the next subsection, we will elaborate on the practical interest~\cite{Lan16,Ripser,Brendel,BLPS13,HMN,HMMNWJD10,LewinerLT04} in computing gradient vector fields on simplicial complexes with (near-)optimal number of critical simplices (unmatched nodes in the Hasse diagram).

\subsection{Motivation}

The idea of using discrete Morse theory to speedup the computation of homology~\cite{HMMNWJD10}, persistent homology~\cite{Ripser, MN13}, and multidimensional persistence~\cite{Lan16} hinges on the fact that discrete Morse theory helps to reduce the problem of computing homology groups of an input simplicial complex to computing homology groups of a smaller collapsed cell or chain complex. In fact, certain state of the art methods for computing homology groups of complexes ~\cite{HMMNWJD10} and persistent homology of filtrations~\cite{Ripser} depend crucially on discrete Morse theory. In particular, in the numerical experiments for homology computation reported in  Harker et al.~\cite{HMMNWJD10}, a discrete Morse theory based preprocessing step led to a speedup of several orders of magnitude over existing methods on a wide variety of datasets. In a followup work, Harker et al.~\cite{HMN}  devised a discrete Morse theory based framework to efficiently compute the induced map on homology, a problem that arises in Conley index computations. More recently, Brendel et al.~\cite{Brendel} designed a discrete Morse theory based algorithm to compute (typically small) presentations of the fundamental group of finite regular CW-complexes, and certain knot invariants.  

Thus, finding near-optimal gradient vector fields is a central problem computational topology, with a wide range of applications.
However, finding an optimal gradient vector field turns out to be an NP-hard problem, as shown by Joswig et al.~\cite{JP06} via a reduction from the \emph{erasability problem} introduced by E\v gecio\v glu and Gonzalez~\cite{EG96}. 

On the other hand, certain heuristics for Morse matching have been reported to be highly effective, often achieving optimality in practice~\cite{HMMNWJD10,JLT14,Le03a}. This naturally raises the question of approximability: to which extent is it feasible to obtain near-optimal solutions for Morse matching in polynomial time? By establishing bounds on the hardness of approximation, we make it evident that for certain instances, the otherwise effective heuristics would fail to compute near-optimal Morse matchings. 

\subsection{The Morse Matching Problems\label{sec:mmprob}}

The Max-Morse Matching problem (\MaxMM) can be described
as follows: Given a simplicial complex $\mathcal{K}$, compute a gradient
vector field that maximizes the cardinality of matched (regular) simplices,
over all possible gradient vectors fields on $\mathcal{K}$. Equivalently,
the goal is to maximize the number of gradient pairs. For the complementary
problem Min-Morse Matching (\MinMM), the goal is to compute
a gradient vector field that minimizes the number of unmatched (critical)
simplices, over all possible gradient vector fields on $\KCC$. While
the problem of finding an exact optimum are equivalent for \MinMM
and \MaxMM, the approximation variants of these problems have vastly
different flavors, as we shall note in Sections~\ref{sec:hardnessmin}~and~\ref{sec:hardnessmax}.

\subsection{Related work\label{sec:relwork}}

Based on the relationship between erasability and Morse Matching observed by Lewiner~\cite{Le02,Le03a},
Joswig et al.~\cite{JP06} established NP-completeness of the Morse Matching Problem, using a reduction that is not approximation preserving, and posing the approximability
of Morse matching as an open problem~\cite[Sec.~4]{JP06}.
The algorithmic question of finding optimal discrete Morse functions
on simplicial complexes is a well studied problem. Most methods so
far have relied on effective heuristics~\cite{JP06,HMMNWJD10, BLW11,He05,Le03a}.
The first theoretical result in context of Morse matchings was established by Burton et al.~\cite{BLPS13}, who developed a fixed parameter tractable algorithm for computing optimal Morse functions on 3-manifolds. More recently, Rathod et al.~\cite{RBN17} proposed the first approximation algorithms for \MaxMM on simplicial complexes that provide constant factor
approximation bounds for fixed dimension. 

Methodologically, the mechanism of collapses employed in our proof of hardness of approximation for \MaxMM bears notable resemblance to sequential collapsing of Bing's house gadgets used by Malgouryes and Franc\'es~\cite{MF}  and Tancer~\cite{Tan16} for proving NP-hardness results for certain collapsibility problems.

\subsection{Our contributions \label{sec:contrib}}

In Section~\ref{sec:hardnessmin}, using Tancer's result~\cite{Tan16} about NP-completeness
of \emph{collapsibility}, we provide a straightforward proof of
inapproximability of \MinMM on simplicial complexes with dimension $d\geq 3$.
In particular, we prove that, assuming $P\neq \mathit{NP}$, 
there is no $O(n^{1-\epsilon})$-factor approximation algorithm
for \MinMM for any $\epsilon>0$, where $n$ denotes the total number
of simplices in a given complex $\KCC$. 
Then, in Section~\ref{sec:hardnessmax}, we prove that, for any $\epsilon>0$, approximating \MaxMM for simplicial complexes of dimension $d\geq2$ within a factor of $
\left(1 - \frac{1}{4914}\right)+\epsilon$ is NP-hard and approximating it within a factor of $
\left(1 - \frac{1}{702}\right)+\epsilon$ is UGC-hard. In particular, this shows that \MaxMM has no PTAS unless $P=\mathit{NP}$. 

\section{Background and Preliminaries \label{sec:Background}}

\subsection{Simplicial complexes} 
\label{sub:Simplicial complexes}

A $k$\emph{-simplex} $\sigma=\conv V$ is the convex hull
of a set $V$ of $(k+1)$ affinely independent points in $\mathbb{R}^{d}$. We call $k$ the dimension of $\sigma$.
We say that $\sigma$ is \emph{spanned} by the points $V$.
Any nonempty subset of $V$
also spans a simplex, a \emph{face} of $\sigma$. $\sigma$ is a \emph{coface}
of $\tau$ iff $\tau$ is face of $\sigma$. 
We say that $\sigma$ is a \emph{facet} of $\tau$ if $\sigma$ is a face of $\tau$ with $\dim\sigma=\dim\tau-1$. 
A \emph{simplicial complex} $\KCC$  is a collection of simplices that satisfies the following conditions:
\begin{compactitem}
\item	any face of a simplex in $\KCC$ also belongs to $\KCC$, and
\item	the intersection of two simplices $\sigma_1,\sigma_2\in \KCC$ is either empty or a face of both $\sigma_1$ and $\sigma_2$. 
\end{compactitem}
The \emph{underlying space} of $\KCC$ is the union of its simplices, denoted by $|\KCC|$. The underlying space is implicitly used whenever we refer to $\KCC$ as a topological space.

An \emph{abstract simplicial complex} $\SCC$ 
is a collection of finite nonempty sets $A \in \SCC$ such that every nonempty subset of $A$ is also contained in $\SCC$.
The sets in $\SCC$ are called its \emph{simplices}.
For example, the vertex sets of the simplices in a geometric complex form an abstract simplicial complex, called its \emph{vertex scheme}.
If $\KCC$ is a geometric simplicial complex whose vertex scheme is isomorphic to an abstract simplicial complex $\SCC$, then $\KCC$ is a \emph{geometric realization} of $\SCC$. It is unique up to simplicial isomorphism.

We will use the construction of a \emph{pasting map}~\cite{Munkres} to perform vertex and edge identifications on simplicial complexes.
Given a finite abstract simplicial complex $\LCC$, a \emph{labelling}
of the vertices of $\LCC$ is a surjective map $f:\LCC^{(0)}\to C$ where the set $C$ is called the set of \emph{vertex labels}.
Then, the set $\{ f(\sigma) \mid \sigma \in \LCC\}$ is an abstract simplicial complex. Let $\KCC$ be a geometric realization. Then $f$ induces a simplicial quotient map $g:\left|\LCC\right|\to\left|\KCC\right|$, called the \emph{pasting map} associated with $f$. 

In particular, given an equivalence relation $\sim$ on the vertices of $\LCC$, the surjection sending each vertex to its equivalence class induces a pasting map. We will use this construction to perform identifications of simplices. For example, the wedge sum of a collection of \emph{pointed simplicial complexes} (complexes with a distinguished vertex, called the \emph{basepoint}) can be constructed this way using the equivalence relation identifying all basepoints.

\subsection{Discrete Morse theory and Erasability}
\label{sub:Discrete-Morse-Theory}

Our focus in this paper is limited to simplicial complexes, and hence
we restrict the discussion of Forman's discrete Morse theory to simplicial
complexes. We refer to~\cite{Fo02} for a compelling expository introduction. 

A function $f$ on a simplicial complex $\KCC$ is called a \emph{discrete Morse function} if
\begin{compactitem}
\item $f$ is \emph{monotonic}, i.e., $\sigma \subseteq \tau$ implies $f(\sigma) \leq f(\tau)$, and
\item for all $t \in \im(f) $,  $f^{-1}(t)$ is either a singleton $\left\{ \sigma\right\}$ (in which case $\sigma$ is a \emph{critical simplex}) or a pair $\left\{\sigma,\tau\right\}$, where $\sigma$ is a facet of $\tau$ (in which case $\left(\sigma,\tau\right)$ form a \emph{gradient pair} and $\sigma$ and $\tau$ are \emph{regular simplices}). 
\end{compactitem}
Given a discrete Morse function $f$ defined on complex $\KCC$, the \emph{discrete gradient vector field} $\VSS$ of $f$ is the  collection  of pairs of simplices  $\left(\sigma,\tau\right)$, where $\left(\sigma,\tau\right)$ is in $\VSS$ if and only if $\sigma$ is a facet of $\tau$ and $f(\sigma)=f(\tau)$.

Discrete gradient vector fields have a useful interpretation in terms of acyclic graphs obtained from matchings on Hasse diagrams, due to Chari~\cite{Ch00}.
Let $\KCC$ be a simplicial complex, let $\HKCC$ be its Hasse diagram, and let $M$ be a matching in the underlying undirected graph $\HKCC$. Let $\HKCC(M)$ be the directed graph obtained from $\HKCC$ by reversing the direction of each edge of the matching $M$. Then $M$ is a \emph{Morse matching} if and only if $\HKCC(M)$	is a directed acyclic graph. Every Morse matching $M$ on the Hasse diagram $\HKCC$ corresponds to a unique gradient vector field $\VCC_M$ on complex $\KCC$ and vice versa. For a Morse matching $M$, the unmatched vertices correspond to critical simplices of $\VCC_M$, and the matched vertices correspond to the regular simplices of $\VSS_M$.

A non-maximal face $\sigma\in\KCC$ is said to be a \emph{free face} if it
is contained in a unique maximal face $\tau\in\KCC$. %
If $\dim\tau=\dim\sigma+1$,
we say that $\KCC^{\prime} = \KCC \setminus \{\sigma,\tau\}$ arises from $\KCC$ by an \emph{elementary
collapse}, denoted by $\KCC\searrow^{e}\KCC^{\prime}$. Furthermore, we say that
$\KCC$ \emph{collapses} to $\LCC$, denoted by $\KCC\searrow\LCC$,
if there exists a sequence $\KCC=\KCC_{1},\KCC_{2},\dots\KCC_{n}=\LCC$
such that $\KCC_{i}\searrow^{e}\KCC_{i+1}$ for all $i$. 
If $\KCC\searrow\LCC$, or more generally, if $\KCC$ and $\LCC$ are related through a sequence
collapses and expansions (inverses of collapses),
then the two complexes are \emph{simple-homotopy equivalent
type}. In particular, $\KCC$ and $\LCC$ are homotopy equivalent. Furthermore, if $\KCC$ collapses to a point, one says that $\KCC$ is collapsible
and writes $\KCC\searrow0$.

A simplicial collapse can be encoded by a discrete gradient.
\begin{thm}[Forman~\cite{Fo98}, Theorem 3.3]
  \label{Gradient Collapsing Theorem}
  Let $\KCC$ be a simplicial complex with a discrete gradient vector field $\VSS$,
  and let $\LCC \subseteq \KCC$ be a subcomplex.
  If $\KCC \setminus \LCC$ is a union of pairs in $\VSS$, then $\KCC \searrow \LCC$.
\end{thm}
In this case, we say that the collapse $\KCC \searrow \LCC$ is \emph{induced by} the gradient $\VSS$. As a consequence of this theorem, we obtain:

\begin{thm}[Forman~\cite{Fo98}, Corollary 3.5]
Let $\KCC$ be a simplicial complex with a discrete gradient vector field $\VSS$ and let $m_d$ denote the number of critical simplices of $\VSS$ of dimension $d$. Then $\KCC$ is homotopy equivalent to a CW complex with exactly $m_d$ cells of dimension $d$.
\end{thm}
In particular, a discrete gradient vector field on $\KCC$ with $m_d$ critical simplices of dimension~$d$ gives rise to a chain complex having dimension $m_d$ in each degree $d$, whose homology is isomorphic to that of $\KCC$. This condensed representation motivates the algorithmic search for (near-)optimal Morse matchings.

We will later use the following elementary lemma about gradient vector fields.
\begin{lem}
\label{lem:uniqueCriticalVertex}
Let $\KCC$ be a connected simplicial complex, let $p$ be a vertex of $\KCC$, and let $\VSS_1$ be a discrete gradient on $\KCC$ with $m_0>1$ critical simplices of dimension $0$ and $m$ critical simplices in total. Then there exists another gradient vector field $\widetilde\VSS$ on $\KCC$ with $p$ as the only critical simplex of dimension $0$ and $m - 2(m_0-1)$ critical simplices  in total. 
\end{lem} 
\begin{pf} Let $\LCC$ be the set of all the $1$-simplices paired with $2$-simplices in $\VSS$.
Let $\KCC_1$ be the $1$-skeleton of $\KCC$. Then, by \cite[Lemma~4.2]{JP06}, the $1$-complex $\KCC_1 \setminus \LCC$ is connected, and one can compute a gradient vector field $\VSS_1$ on $\KCC_1 \setminus \LCC$ with $p$ as the single critical $0$-simplex 
using depth first search starting from $p$ (see, e.g., \cite{RBN17}). Let $\WSS \subset \VSS$ consist of all gradient pairs of $\VSS$ contained in $\KCC \setminus (\KCC_1 \setminus \LCC) = (\KCC \setminus \KCC_1) \cup \LCC$. Note that, by construction, $\WSS$ does not contain any pairs of dimensions $(0,1)$, while  $\VSS_1$ has only such pairs.
Since the gradient vector fields $\VSS_1$ and $\WSS$ are defined on disjoint sets of simplices, it follows that $\widetilde\VSS = \VSS_1\cup \WSS$ is a gradient vector field with the desired property.
\end{pf}

\begin{cor} \label{cor:unicollapse} Given a collapsible simplicial complex $\KCC$ and an arbitrary vertex $p \in \KCC$, there exists a gradient vector field $\VSS$ on $\KCC$ with $p$ as the unique critical simplex of $\VSS$.
\end{cor}

Borrowing and extending the terminology
used in~\cite{EG96}, we make the following definitions:
A maximal face $\tau$ in a simplicial complex $\KCC$ is called an \emph{internal simplex} if it has no free face. 
If a $2$-complex $\KCC$ collapses to a $1$-complex, we say that $\KCC$ is \emph{erasable}. Moreover, for a $2$-complex $\KCC$, the quantity $\erk$ is the minimum number
of internal $2$-simplices that need to be removed so that the resulting complex collapses to a $1$-complex. Equivalently, it is the minimum number of critical $2$-simplices of any discrete gradient on $\KCC$.
For a complex $\KCC$, we denote the set of $d$-simplices of $\KCC$ by $\dcomplex{\KCC}{d}$. 

\begin{defn}[Erasable subcomplex of a complex] Given a $2$-complex $\KCC$, we say that a subcomplex $\LCC \subseteq \KCC$ is an \emph{erasable subcomplex of $\KCC$ (through the gradient $\VSS$)} if there exists another subcomplex $\MCC \subseteq \KCC$ with $\KCC \searrow \MCC$ (induced by the gradient $\VSS$) such that $\KCC \setminus \MCC \subseteq \LCC$ and 
$\twocomplex\KCC \setminus \twocomplex\MCC = \twocomplex\LCC$.
\end{defn}

\begin{defn}[Eventually free] We say that a simplex $\sigma$ is \emph{eventually free (through the gradient $\VSS$)} in a complex $\KCC$ if there exists a subcomplex $\LCC$ of $\KCC$ such that $\KCC \searrow \LCC$ (induced by $\VSS$) and $\sigma$ is free in $\LCC$.
Equivalently, $\KCC$ collapses further to a subcomplex not containing~$\sigma$.
\end{defn}

\begin{lem}
\label{lem:unionErasable}
If $\LCC_1,\LCC_2$ are erasable subcomplexes of a 2-complex $\KCC$, then so is their union.
\end{lem}
\begin{proof}
Let $\VSS_1$ be a discrete gradient erasing $\LCC_1$, and $\VSS_2$ a discrete gradient erasing $\LCC_2$.
Without loss of generality, we may assume that both gradients have only pairs $(\sigma,\tau)$ of dimension $(1,2)$, and that all such pairs are in $\LCC_1$ or $\LCC_2$, respectively, and $\sigma$ is eventually free; removing all other pairs still yields an erasing gradient.
Now consider the collapse $\KCC \searrow \MCC_1$ induced by $\VSS_1$ and the collapse $\KCC \searrow \MCC_2$ induced by $\VSS_2$. 
Restricting the gradient $\VSS_2$ to the subcomplex $\MCC_1$, we obtain
a gradient $\VSS_{12} = \{(\sigma,\tau) \in \VSS_2 \mid \sigma, \tau \in \MCC_1\}$. 
By induction, 
each~$\sigma$ appearing in such a pair $(\sigma,\tau)$ is eventually free in $\MCC_1$, since any $2$-simplex $\psi \in \MCC_1$ that is a coface of $\sigma$ other than $\tau$ must appear in a pair $(\phi,\psi) \in \VSS_{12}$ by the definition of $\VSS_{12}$ and the assumption that $\phi$ is eventually free through the gradient $\VSS_2$.
Thus, $\KCC$ collapses to a complex that contains no $2$-simplices of either $\LCC_1$ or $\LCC_2$, as claimed.
\end{proof}

\subsection{Approximation algorithms}

An $\alpha$-\emph{approximation algorithm} for an optimization problem is a polynomial-time algorithm that, for all instances of the problem, produces a solution whose value is within a factor $\alpha$ of the value of an optimal solution. The factor $\alpha$
is called the \emph{approximation ratio} of the algorithm. An \emph{approximation preserving reduction} is a procedure for transforming an optimization problem $A$ to an optimization problem $B$, such that an $\alpha$-approximation algorithm for $B$ implies an $h(\alpha)$-approximation algorithm for $A$, for some function $h$. Then, if $A$ is hard to approximate within factor $h(\alpha)$, the reduction implies that $B$ is hard to approximate within factor $\alpha$.

We will use a particular important and well-studied class of approximation preserving reductions, called \emph{L-reductions}, which provide a simple and effective tool in proving hardness of approximability results~\cite{PY91, WS10}. 
To give the definition, consider a maximization problem $A$ 
with a non-negative integer valued objective function $m_{A}$.
Given an instance $x$ of $A$, the goal is to find a solution~$y$ (among a finite set of feasible solutions) maximizing the
objective function $m_{A}(x,y)$. Define $\opa(x)$ as
the maximum value of the objective function on input $x$.   

An \emph{L-reduction}
from one optimization problem $A$ to another optimization problem
$B$ is a pair of functions $f$ and $g$ that are computable in polynomial
time and satisfy the following conditions:
\begin{enumerate}
\item The function $f$ maps instances of $A$ to instances of $B$.
\item There is a positive constant $\mu$ such that, for all instances $x$ of $A$, \[\opb(f(x))\leq \mu \, \opa(x) . \]
\item The function $g$ maps instances of $A$ and solutions of $B$ to solutions of $A$.
\item There is a positive constant $\nu$ such that, for any instance $x$ of $A$ and any solution $y$ of $f(x)$, we have
\[\opa(x) - m_A(x,g(x,y)) \leq \nu \left( \opb(f(x)) - m_B(f(x),y) \right) . \]
\end{enumerate}
If $\mu=\nu=1$, the reduction is \emph{strict}.

We will use the following straightforward fact about L-reductions for proving hardness of approximation bounds. 
\begin{thm}[Williamson and Shmoys, \cite{WS10}, Theorem~16.5]\label{thm:newfactor} If there is an L-reduction with parameters $\mu$ and $\nu$ from a maximization problem $A$ to another maximization problem $B$, and there is a $(1-\delta)$-approximation algorithm for $B$,
then there is a $(1-\mu\nu\delta)$-approximation algorithm for $A$.
\end{thm}

\subsection{Acyclic subgraphs}

We recall some concepts and problems from graph theory that will be used in our reductions. A directed graph $\GCC$ with vertex set $\VCC$ and edge set $\ECC$
is written as $\GCC=(\VCC,\ECC)$. 
A directed graph is called an \emph{oriented graph} if no pair of vertices is connected by an anti-parallel pair of edges.
In other words, an oriented graph is a directed graph without $2$-cycles or loops.
Note that in contrast to a general directed graph, an oriented graph always has a simple \emph{underlying undirected graph}, which is therefore a simplicial complex. We will be making use of this fact in \cref{sec:hardnessmax}. 

The problem of finding the \emph{maximum acyclic subgraph} (\MAS) of a given directed graph $\GCC=(\VCC,\ECC)$ consists of determining a maximum subset $\ECC_{\mathit{max}}\subseteq\ECC$ for which the subgraph $\GCC_{\mathit{max}}=(\VCC,\ECC_{\mathit{max}})$ has no directed cycles. A \emph{feedback arc set} is a set of edges whose removal leaves a directed acyclic graph. A \emph{minimum feedback arc set} is a feedback arc set of minimum cardinality. The problem \minFAS of finding such a set is thus complementary to \MAS.

A \emph{directed degree-3 graph} is a directed graph with total degree (indegree plus outdegree) at most $3$. 
The restriction of the problem \MAS to directed degree-3 graphs is denoted by \threeMAS.
Moreover, the problem \MAS restricted to oriented graphs is denoted by \OMAS, and the restriction to oriented degree-3 graphs is denoted by \threeOMAS.

We will show that there is a L-reduction from \MAS to \OMAS, allowing us to consider only oriented graphs later.

\begin{thm}
\label{thm:MAStoOMAS}
There is a strict reduction
from \MAS to \OMAS,
and from \threeMAS to \threeOMAS.
\end{thm}

\begin{proof}
The map $f$ transforming an instance of \MAS (a directed graph $\GCC$) to an instance of \OMAS (an oriented graph $f(\GCC)$) is given by removing from $\GCC$ all loops and all pairs of anti-parallel edges.
Furthermore, the map $g$ transforming a solution of \OMAS for the instance $f(\GCC)$ (an acyclic subgraph $\ACC$ of $f(\GCC)$) to a solution of \MAS for the instance $\GCC$ (an acyclic subgraph $\BCC = g(\GCC,\ACC)$ of $\GCC$) is given as follows: 
Extend the acyclic graph $\ACC$ to a subgraph $\BCC$ of $\GCC$ by adding for each anti-parallel pair of edges in $\GCC$ one edge whose orientation is consistent with the partial order induced by $\ACC$. By construction, the subgraph $\BCC$ is still acyclic.

Let $e$ be the number of edges in $\GCC$, let $k$ be the number of pairs of anti-parallel edges in $\GCC$, and let $a$ be the number of edges in $\ACC$. Then the number of edges in $\BCC$ is $a+k$. On the other hand, any acyclic subgraph $\CCC$ of $\GCC$ restricts to an acyclic subgraph of $f(\GCC)$ by removing at most $k$ edges.
Thus we have
\begin{align*}
m_\MAS(\GCC,g(\GCC,\ACC)) &= m_\OMAS(f(\GCC),\ACC) + k, \\
\opt\MAS(\GCC) &= \opt\OMAS(f(\GCC)) + k
\intertext{and thus $\opt\OMAS(f(\GCC)) \leq \opt\MAS(\GCC)$ and}
\opt\MAS(\GCC) - m_\MAS(\GCC,g(\GCC,\ACC)) &= \opt\OMAS(f(\GCC)) - m_\OMAS(f(\GCC),\ACC),
\end{align*}
establishing a strict reduction from \MAS to \OMAS.
The same construction restricts to a strict reduction from \threeMAS to \threeOMAS.
\end{proof}

We state a few known hardness of approximation results for \MAS and related problems.
\begin{thm}[Newman~\cite{New},~Theorem~3] \label{thm:3masnp} It is NP-hard to approximate \threeMAS to within $\big(1-\frac{1}{126}\big)+\epsilon$ for any $\epsilon>0$.
\end{thm}

Moreover, the following result establishes hardness with respect to the \emph{unique games conjecture} (UGC) \cite{Khot2002Power}.  
A problem is said to be \emph{UGC-hard} (or \emph{UG-hard}) if the unique games conjecture implies that the problem is NP-hard.
We refer to~\cite{Khot2010UGC} for a detailed account on this conjecture.

\begin{thm}[Guruswami et al.~\cite{Guru}, Theorem 1.1] \label{thm:masugc} Let $\delta \in (0, \frac{1}{2})$. If for any directed graph $G$ with an acyclic subgraph consisting of a fraction $(1-\delta)$ of its edges, one can efficiently find an acyclic subgraph of G with more than $(\frac{1}{2} + \delta)$ of its edges, then the UGC  is false.
In particular, it is UGC-hard to approximate \MAS
within a factor of $\frac{1}{2}+\delta$ for any~$\delta>0$.
\end{thm}

 By \cref{thm:MAStoOMAS}, the same is true for \OMAS. Moreover, Newman~\cite{New} established an approximation preserving reduction from \MAS to \threeMAS, with the following consequence:

\begin{thm}[Newman~\cite{New}, Theorem 5]\label{thm:3masugc} For any constant $\epsilon>0$, if there exists a $\left(\left(1 - \frac{1}{18}\right)+\epsilon\right)$-approximation algorithm for \threeMAS, then there exists a $\left(\frac{1}{2}+\delta\right)$-approximation algorithm for \MAS for some constant $\delta>0$.
\end{thm}

From \cref{thm:MAStoOMAS,thm:3masnp,thm:masugc,thm:3masugc}, we conclude:
\begin{cor} \label{cor:hammer}
It is UGC-hard to approximate \threeMAS and \threeOMAS within a factor of $\left(1 - \frac{1}{18}\right)+\epsilon$, and NP-hard to approximate \threeMAS and \threeOMAS to within $\left(1 - \frac{1}{126}\right)+\epsilon$, for any $\epsilon > 0$.
\end{cor}

\section{Hardness of Approximation of Min-Morse Matching}
\label{sec:hardnessmin}

In this section, we work with abstract connected simplicial complexes. Recall that an abstract simplicial complex is connected if its $1$-skeleton is connected as a graph. 

\begin{defn}[Amplified complex] Given a pointed simplicial complex $\KCC$ with $n$ simplices and some integer $c>0$, the amplified complex $\widehat{\KCC}_{c}$ is defined as the wedge sum of $m$ copies of~$\KCC$, with $m = n^{c-1}$. 
\end{defn}

\begin{lem} \label{lem:blowup} Given a complex $\KCC$ of size $n$ and integer $c$, consider the amplified complex $\widehat{\KCC}_{c}$ of $\KCC$. Let $\widehat\VSS$ be a gradient vector field on $\widehat{\KCC}_{c}$. Then
\begin{enumerate}[(i)]
\item $\widehat{\KCC}_{c}$ is collapsible if and only if $\KCC$ is collapsible. 
\item If $\KCC$ is not collapsible, then $\widehat\VSS$ has more than $n^{c-1}$ critical simplices.
\end{enumerate} 
\end{lem}
\begin{pf} Suppose that the complex $\KCC$ is collapsible. Then there exists a gradient vector field $\VSS_\KCC$ with a unique critical simplex $q\in\KCC$. Let $p$ be an arbitrarily chosen distinguished vertex of $\KCC$ that will be used to construct amplified complex $\widehat{\KCC}_{c}$. Using \cref{cor:unicollapse}, without loss of generality the vector field $\VSS_\KCC$ has $p$ as its unique critical simplex. Now the gradient vector field on $\widehat{\KCC}_{c}$, say $\widehat\VSS$, is simply the gradient vector field $\VSS_\KCC$ repeated on each identical copy of $\KCC$. Since $p$ is the unique critical simplex of $\widehat\VSS$ on $\widehat{\KCC}_{c}$, we conclude that $\widehat{\KCC}_{c}$ is collapsible. 

Conversely, suppose that the complex $\widehat{\KCC}_{c}$ is collapsible. Then, by \cref{cor:unicollapse}, we can obtain a gradient vector field $\widehat{\VSS}$ on $\widehat{\KCC}_{c}$ with the distinguished vertex $p$ as its unique critical simplex. If we consider the gradient vector field $\widehat{\VSS}$ restricted to any one of the copies of $\widehat{\KCC}_{c}$, it follows immediately that $\KCC$ is collapsible.

Now suppose that $\KCC$ is not collapsible and $\widehat\VSS$ has less than or equal to $n^{c-1}$ critical simplices. By \cref{lem:uniqueCriticalVertex}, without loss of generality we may assume that $\widehat\VSS$ has the distinguished vertex $p$ as the unique critical $0$-simplex. Now consider $\widehat\VSS$ restricted to each of the individual copies of $\KCC$. Then clearly at least one of the copies has $p$ as its unique critical simplex (else we would have more than $n^{c-1}$ critical simplices in total). But this immediately implies that $\KCC$ is collapsible, a contradiction. Hence, if $\KCC$ is not collapsible, then  $\widehat\VSS$ has more than $n^{c-1}$ critical simplices.
\end{pf}

\begin{prop} \label{prop:magicalgorithm} For any $\epsilon \in (0,1]$, if there exists an $O(n^{1-\epsilon})$-factor approximation algorithm for \MinMM, where $n$ denotes the number of simplices of an input simplicial complex, then there exists a polynomial time algorithm for deciding collapsibility of simplicial complexes.
\end{prop}
\begin{pf} Given any $\epsilon \in (0,1]$, suppose there exists an $O(n^{1-\epsilon})$-factor approximation algorithm for \MinMM. Specifically, there exist $p, M>0$ such that for all $n \geq M$, the approximation ratio is bounded above by $p n^{1-\epsilon}$.  
Now choose $c$ to be the smallest positive integer with the property $\frac{1}{c} < \epsilon$, i.e., $c = \left\lfloor \frac{1}{\epsilon} + 1 \right\rfloor$. Consider an arbitrary connected complex $\KCC$ with $n$ simplices and construct the amplified complex $\widehat{\KCC}_{c}$. Note that the total number of simplices in $\widehat{\KCC}_{c}$ is $\hat n = \left( n - 1 \right) n^{c-1} + 1 = n^c - n^{c-1} + 1$.
Also, if  $n < \max\{M,p^{\frac{1}{1+c\epsilon}}\}$, i.e., if $n$ is bounded by a constant, the collapsibility of $\KCC$ can easily be checked in constant time. So, without loss of generality, we assume that $n \geq \max\{M,p^{\frac{1}{1+c\epsilon}}\}$.

We now use the following Algorithm~B to decide collapsibility of the complex $\KCC$. We execute the $O(n^{1-\epsilon})$-factor approximation Algorithm~A for \MinMM on the amplified complex $\widehat{\KCC}_{c}$. If the number $C_A$ of critical simplices returned by Algorithm~A is less than $n^{c-1}$, we report that the complex $\KCC$ is collapsible, else we declare that $\KCC$ is not collapsible.

When the complex is collapsible, the number $C_A$ of critical simplices returned by Algorithm~A can be bounded as follows:
\[
C_A  \leq   p \hat n^{1-\epsilon} <  p\left(n^{c}\right)^{1-\epsilon} \leq  n^{c-1} .
\]
The bound $C_A < n^{c-1}$ for a collapsible complex $\KCC$ with $n$ simplices, along with part (ii) of Lemma~\ref{lem:blowup}, establishes the correctness of Algorithm~B for determining collapsibility of the complex $\KCC$. Also, since $\hat n <  n^{c}$, Algorithm~B runs in time polynomial in $n$.
\end{pf} 

Recently, Tancer~\cite{Tan16} proved the following theorem about
collapsibility of $3$-complexes.

\begin{thm}[Tancer~\cite{Tan16}, Theorem 1] \label{thm:tancer} It is NP-complete to decide whether a given $3$-complex is collapsible.
\end{thm}

\begin{cor}  For any $\epsilon \in (0,1]$, there exists no $O(n^{1-\epsilon})$-factor approximation algorithm for \MinMM, where $n$ denotes the number of simplices of an input simplicial complex, unless $P = \mathit{NP}$.
\end{cor}

\section{Hardness of Approximation for Max-Morse Matching}
\label{sec:hardnessmax}

In this section, we describe an L-reduction from \threeOMAS to \MaxMM, establishing hardness of approximation for \MaxMM. Our construction is based on a modification of Zeeman's \emph{dunce hat}~\cite{Zeeman}. The dunce hat is a simplicial complex which is contractible but has no free faces and is therefore not collapsible. 
In contrast, the \emph{modified dunce hat} is collapsible but only through a single free face. 
The triangulation is given in \cref{fig:cellAttach}. An equivalent triangulation has been described by Hachimori \cite[p.~108]{Hachimori2000Combinatorics}. Its number of simplices is minimal among all complexes that are collapsible through a single free face, as can be verified by an exhaustive search~\cite{Patak2018Personal}. The 1-simplex $\omega$ is the unique free face of the modified dunce hat $\DCC$.

\begin{figure}[h]
\begin{center}

\definecolor{verylightgray}{rgb}{0.95,0.95,0.95}
\definecolor{verylightblue}{rgb}{0.925,0.95,1.0}

\scriptsize
\begin{tikzpicture}[scale=0.6,baseline]%
    \coordinate (r) at (0,-1);
    \coordinate (w) at (-1,0);
    \coordinate (t) at (0,1);
    \coordinate (v) at (1,0);
    \coordinate (s) at (2,-2);
    \coordinate (u) at (2,2);
    \coordinate (sb) at (-2,2);
    \coordinate (q) at (-2,0);
    \coordinate (ur) at (-2,-2);
    \coordinate (ql) at (2,0);
    \coordinate (qr) at (0,-2);
     
    \draw[fill=verylightgray] (u)--(s)--(ur)--(sb)--cycle;
    \draw[fill=verylightblue] (u)--(sb)--(t)--cycle;
    \draw (r)--(v);
    \draw[Red,thick] (sb)--(u);
    \draw (v)--(u);
    \draw[Green,thick] (t)--(sb);
    \draw[Red,thick] (t)--(v);
    \draw[Red,thick] (t)--(w);
    \draw (w)--(q);
    \draw (w)--(ur);
    \draw (w)--(sb);
    \draw (r)--(ur);
    \draw (r)--(qr);
    \draw (s)--(r);
    \draw (s)--(v);
    \draw (ql)--(v);
    \draw (u)--(t);
    \draw (r)--(t);
    \draw (r)--(w);
    
    \node[left] at (r) {$r$};
    \node[right] at (w) {$w$};
    \node[above] at (t) {$t$};
    \node[left] at (v) {$v$};
    \node[below right] at (s) {$s$};
    \node[above right] at (u) {$u$};
    \node[above left] at (sb) {$s$};
    \node[left] at (q) {$q$};
    \node[below left] at (ur) {$u$};
    \node[right] at (ql) {$q$};
    \node[below] at (qr) {$q$};
    \node[above,Red] at ($(u)!1/2!(sb)$) {$\omega$};
    \node[below left=-1mm and 0mm,Green] at ($(t)!1/2!(sb)$) {$\eta$};
    \node[above right=-1mm and -1mm,Red] at ($(t)!1/2!(v)$) {$\phi$};
    \node[above left=-1mm and -1mm,Red] at ($(t)!1/2!(w)$) {$\psi$};
    \node[right=1mm,RoyalBlue] at ($(u)!1/2!(sb)!1/3!(t)$) {$\Gamma$};

\end{tikzpicture}
\hfil
\begin{tikzpicture}[scale=0.6,baseline]%
    \coordinate (r) at (0,-1);
    \coordinate (w) at (-1,0);
    \coordinate (t) at (0,1);
    \coordinate (v) at (1,0);
    \coordinate (s) at (2,-2);
    \coordinate (u) at (2,2);
    \coordinate (sb) at (-2,2);
    \coordinate (q) at (-2,0);
    \coordinate (ur) at (-2,-2);
    \coordinate (ql) at (2,0);
    \coordinate (qr) at (0,-2);

    \draw[fill=verylightgray] (u)--(s)--(ur)--(sb)--cycle;
    \draw (r)--(v);
    \draw (sb)--(u);
    \draw (v)--(u);
    \draw[Red,thick] (t)--(sb);
    \draw (t)--(v);
    \draw (t)--(w);
    \draw (w)--(q);
    \draw (w)--(ur);
    \draw (w)--(sb);
    \draw (r)--(ur);
    \draw (r)--(qr);
    \draw (s)--(r);
    \draw (s)--(v);
    \draw (ql)--(v);
    \draw (u)--(t);
    \draw (r)--(t);
    \draw (r)--(w);
    
    \draw[very thick,->,CornflowerBlue] ($(u)!1/2!(sb)$)--($(u)!1/2!(sb)!1/3!(t)$);
    \draw[thick,->] ($(u)!1/2!(t)$)--($(u)!1/2!(t)!1/3!(ql)$);
    \draw[thick,->] ($(u)!1/2!(v)$)--($(u)!1/2!(v)!1/3!(ql)$);
    \draw[thick,->] ($(s)!1/2!(v)$)--($(s)!1/2!(v)!1/3!(r)$);
    \draw[thick,->] ($(r)!1/2!(v)$)--($(r)!1/2!(v)!1/3!(t)$);
    \draw[thick,->] ($(v)!1/2!(ql)$)--($(v)!1/2!(ql)!1/3!(s)$);
    \draw[thick,->] ($(r)!1/2!(s)$)--($(r)!1/2!(s)!1/3!(qr)$);
    \draw[thick,->] ($(r)!1/2!(t)$)--($(r)!1/2!(t)!1/3!(w)$);
    \draw[thick,->] ($(r)!1/2!(qr)$)--($(r)!1/2!(qr)!1/3!(ur)$);
    \draw[thick,->] ($(r)!1/2!(ur)$)--($(r)!1/2!(ur)!1/3!(w)$);
    \draw[thick,->] ($(w)!1/2!(ur)$)--($(w)!1/2!(ur)!1/3!(q)$);
    \draw[thick,->] ($(w)!1/2!(q)$)--($(w)!1/2!(q)!1/3!(sb)$);
    \draw[thick,->] ($(sb)!1/2!(w)$)--($(sb)!1/2!(w)!1/3!(t)$);
    \draw[thick,->] (u)--($(u)!1/2!(ql)$);
    \draw[thick,->] (ur)--($(ur)!1/2!(qr)$);
    \draw[thick,->] (ur)--($(ur)!1/2!(q)$);
    \draw[thick,->] (ql)--($(ql)!1/2!(s)$);
    \draw[thick,->] (qr)--($(qr)!1/2!(s)$);
    \draw[thick,->] (q)--($(q)!1/2!(sb)$);
    \draw[thick,->] (r)--($(r)!1/2!(w)$);
    \draw[very thick,->,CornflowerBlue] (v)--($(v)!1/2!(t)$);
    \draw[very thick,->,CornflowerBlue] (w)--($(w)!1/2!(t)$);

\end{tikzpicture}
\caption{The modified dunce hat $\DCC$. Left: triangulation, with certain distinguished simplices highlighted. 
Note that the vertices $s,q,u$ appearing multiple times are identified accordingly.
The complex is collapsible through the unique free face $\omega$.
Right: a discrete gradient $\VSS_\DCC$ on $\DCC$ that leaves only the vertices $s,t$ and the edge $\eta$ critical. The three highlighted arrows (blue, thick) correspond to pairs in $\VSS_\DCC$ that will in some cases be discarded when assembling a gradient on the entire complex $\KG$.}
\label{fig:cellAttach}
\end{center}
\end{figure}
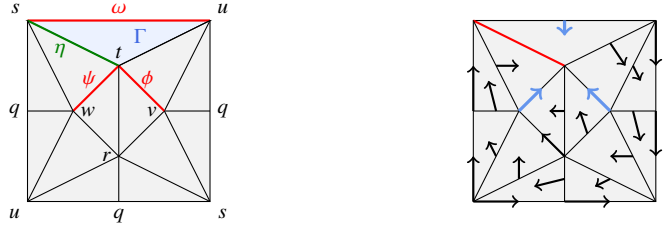

For the remainder of the section, we will use the following notation: For a graph $\GCC=(\VCC,\ECC)$, the indegree of a vertex $v$, in $\GCC$, is denoted $\inDG{v}$, and its outdegree is denoted $\outDG{v}$. 
We now construct a complex $\KG$ from $\GCC$, and more generally, a complex $\KGH$ for any subgraph $\HCC = \left(\vertH,\edgeH\right)$ of $\GCC$.
Throughout this section, the graph $\GCC$ denotes an oriented connected degree-3 graph. Note that the connectedness assumption is not a restriction for the \threeOMAS problem.

In order to aid the reader's intuition, we first outline the motivation behind some of the design choices for the gadget, before giving the formal description. 
Our aim is to construct a complex  $\KG$ so that there is a one-to-one correspondence between the cardinality of a (minimum) feedback arc set~(\cref{prop:erkprop}) and the (minimum) number of simplices that need to be removed from $\KG$ to make it erasable. 
To this end, we start with a disjoint union of copies of $\DCC$, one for each edge $e\in\edgeH$, and make identifications of vertices to obtain a complex $\KG$ which is homotopy equivalent to the undirected graph underlying $\GCC$.
However, we make additional identifications of edges to ensure that any subcomplex $\KGH$, corresponding to a subgraph~$\HCC$, collapses to a 1-dimensional complex isomorphic to the undirected graph underlying $\HCC$ if and only if $\HCC$ is acyclic.
In order to obtain an L-reduction, we ensure that the size of the complex $\KG$ is linear in size of graph $\GCC$.

We now describe the construction of the complexes $\KGH$ and $\KG$ for an oriented degree-3 graph $\GCC$, by classifying the vertices of $\HCC$ into five types.

\begin{enumerate}
\item Consider an arbitrary
total order $\prec$ on the edge set $\ECC$ of $\GCC$.
\item Start with a disjoint union of copies of $\DCC$, one for each edge $e\in\edgeH$, denoted by $\DCC_e$.
\item Using the pasting map construction as defined in \cref{sub:Simplicial complexes}, construct the complex $\KGH$ by identifying some of the distinguished simplices of each gadget
$\DCC_{e}$ based on the following rules, as applied to each vertex $v\in \VCC_\HCC$ according to its indegree and outdegree:

\begin{enumerate}

\item For every vertex $v\in\vertH$
with $\inDH{v}=\outDH{v}=0$, no identifications are made.

\item For every vertex $v\in\vertH$
with $\inDH{v}=0$ and $\outDH{v}>0$, identify all 0-simplices $s_{e}$ for every
outgoing edge $e \in \edgeH$. 

\item For every vertex $v\in\vertH$
with $\outDH{v}=0$ and $\inDH{v}>0$, identify all 0-simplices $t_{e}$ for
every incoming edge $e \in \edgeH$. 

\item For every vertex $v\in\vertH$
with $\inDG{v}=1$, $\outDG{v}=2$, $\inDH{v}>0$ and $\outDH{v}>0$,
let $k$ and $l$ denote the outgoing edges of $v$ in $\GCC$, with $k \prec l$, and let $j$
denote the incoming edge of $v$ in $\GCC$.
\begin{enumerate}
\item 
If $j,k \in \edgeH$, 
identify the 1-simplices $\phi_{j}\sim\omega_{k}$ such that the incident 0-simplices are identified as
$u_{k} \sim v_{j}$ and $s_{k} \sim t_{j}$. 
\item Similarly, if $j,l \in \edgeH$, identify $\psi_{j}\sim\omega_{l}$ such that $u_{l}\sim w_{j}$
and $s_{l}\sim t_{j}$.
\end{enumerate}

\item For every vertex $v\in\vertH$
with $\inDG{v} \in \{1,2\}$, $\outDG{v}=1$, $\inDH{v}>0$ and $\outDH{v}>0$,
let $k$ (and possibly $l$, with $k \prec l$) denote the incoming edges of $v$ in $\GCC$,
and let $j$ denote the outgoing edge of $v$ in $\GCC$.
\begin{enumerate}
\item 
If $j,k \in \edgeH$, 
identify the 1-simplices $\omega_{j}\sim\phi_{k}$ such that $u_{j}\sim v_{k}$
and $s_{j}\sim t_{k}$.
\item (Similarly, 
if $j,l \in \edgeH$, identify $\omega_{j}\sim\phi_{l}$ such that $u_{j}\sim v_{l}$
and $s_{j}\sim t_{l}$.)
\end{enumerate}

\end{enumerate}
\item Furthermore, we define $\KCC(\GCC) = \KCC(\GCC,\GCC)$.
\end{enumerate}

\begin{rmk}
We choose an arbitrary linear order since there is no natural choice to determine some of the attachments in the construction. Using an (arbitrarily chosen) linear order on the edge set $\ECC$ of  $\GCC$ allows us to make the construction of complexes $\KGH$ explicit and concrete. While it is clear that different linear orders on the edge set $\ECC$ may result in different complexes, the hardness results in this section do not depend on the choice of the linear order $\prec$. 
\end{rmk}

\begin{rmk}
Since $\GCC$ is an oriented graph, it is easy to verify that $\KGH$ is a simplicial complex. Also, by construction, $\KGH$ is a subcomplex of $\KG$ whenever $\HCC$ is a subgraph of $\GCC$.
\end{rmk}

\begin{rmk} \label{rem:twoincident} From~\cref{fig:cellAttach}, we observe that for a modified dunce hat $\DCC_{e}$, $\psi_{e}$ is incident on exactly two 2-simplices of $\DCC_{e}$. The same holds true for $\phi_{e}$. Also, note that $\omega_{e}$ is incident on a unique 2-simplex of $\DCC_{e}$, namely $\Gamma_{e}$.
\end{rmk}

\subsection{Structural properties of the reduction}

\begin{lem} \label{lem:erasefree} For a subgraph $\HCC = \left(\vertH,\edgeH\right)$ of a directed degree-3 graph $\GCC$ and an edge $e\in\edgeH$:
\begin{enumerate}[(i)]
\item If $\omega_e$ is eventually free in $\KGH$, then $\DCC_e$ is erasable in  $\KGH$. 
\item If $\DCC_e$ is erasable in $\KGH$ through a gradient $\VSS$, then $\left( \omega_e, \Gamma_e  \right )$ is a gradient pair in $\VSS$. 
If $f$ is a discrete Morse function with gradient $\VSS$, then for any simplex $\sigma \in \DCC_e$ such that $\sigma \notin \{ \omega_e, \Gamma_e \}$ we have $f(\omega_e) > f(\sigma)$.  
\end{enumerate}        
\end{lem}          
\begin{pf} Suppose $\omega_e$ is eventually free in $\KGH$. Then there exists a subcomplex $\LCC$ of $\KGH$ such that $\KGH\searrow \LCC$ and $\omega_e$ is free in $\LCC$.
Note that, by construction of $\DCC$, this implies that $\DCC_e$ is a subcomplex of $\LCC$.
Now using the gradient specified in \cref{fig:cellAttach}
all the 2-simplices of $\DCC_e$ can be collapsed, making $\DCC_e$ erasable in $\KGH$. This proves the first statement of the lemma. The second statement of the lemma immediately follows from observing that $\omega_e$ is the unique free $1$-simplex in complex $\DCC_e$, $\Gamma_e$ is the unique coface incident on $\omega_e$, and $\DCC_e$ is erasable in $\KGH$ through the gradient $\VSS$ of $f$.   
\end{pf}         

\begin{lem} \label{lem:erasehelper} For a subgraph $\HCC$ of a directed degree-3 graph $\GCC$ and a vertex $v \in \vertH$ with $\outDH{v} = l > 0$ and outgoing edges $\left\{ f_1,\dots, f_l \right\} \in \edgeH$, we have:
\begin{enumerate}[(i)]
\item If $\inDH{v} = 0$, then each $\omega_{f_j}$ is free in~$\KGH$.
\item If $\inDH{v} = k>0$, let $\left\{ e_1,\dots, e_k \right\} \in \edgeH$ be the set of incoming edges of $v$. If there is a gradient $\VSS$ such that each $\omega_{e_i}$ is eventually free in $\KGH$ through $\VSS$, then each $\omega_{f_j}$ is eventually free in~$\KGH$ through $\VSS$ as well.
\end{enumerate}
\end{lem}
\begin{pf} If $\inDH{v} = 0$, then each $\omega_{f_j}$ is free by construction of~$\KGH$.
Now suppose that $\inDH{v} = k>0$ and each $\omega_{e_i}$ is eventually free in $\KGH$.  
From Remark~\ref{rem:twoincident} and from the construction of the complex $\KGH$, it can be deduced that for any edge $f_j$, the only $2$-simplices incident on $\omega_{f_j}$ in the complex $\KGH$ are $\Gamma_{f_j}$ and one pair of $2$-simplices of $\DCC_{e_i}$ for each $e_i$.
By assumption, each $\omega_{e_i}$ is eventually free, and so by part (i) of \cref{lem:erasefree}, each $\DCC_{e_i}$ is erasable. Hence, by \cref{lem:unionErasable}, their union is erasable too. This means that
$\KGH$ collapses to a complex $\MCC$ 
in which each $\omega_{f_j}$ is free, proving the claim.
\end{pf} 

\begin{lem} \label{lem:acycliceqerase} A subgraph $\HCC$ of a directed degree-3 graph $\GCC$ is acyclic if and only if the corresponding complex $\KGH$ is erasable. \end{lem}
\begin{pf} Suppose that the given subgraph $\HCC$ of $\GCC$, with $n$ vertices , is acyclic. Consider an arbitrary total order on the vertices of $\HCC$ consistent with the partial order induced by the edges in $\HCC$, and index the vertices $\{v_1,v_2,\dots,v_n\}$ according to this total order. 
We can now apply Lemma~\ref{lem:erasehelper} and part (i) of Lemma~\ref{lem:erasefree} inductively for all $v_i$ from $v_1$ to $v_{n-1}$ to establish the erasability of $\DCC_{f_j}$ for each of the outgoing edges $f_j$ of $v_i$ in $\HCC$. Hence, the entire complex $\KGH$ is erasable.

To show the reverse implication, we prove that if $\HCC$ has directed cycles, then $\KGH$ is not erasable. Assume for a contradiction that $\KGH$ is erasable through a gradient $\VSS$, and let~$f$ be a discrete Morse function with that gradient.
Let $a,b$ be two consecutive edges in a directed cycle of $\KGH$.
Then, by construction of $\KGH$, either $\phi_a \sim \omega_b$ or $\psi_a \sim \omega_b$, and so by part (ii) of \cref{lem:erasefree} we have $f(\omega_a) > f(\omega_b)$. Applying this argument to each pair of consecutive edges in the cycle yields a contradiction.
Hence, if $\HCC$ has directed cycles, then $\KGH$ is not erasable.
\end{pf}

\begin{lem} \label{lem:remainerase} For any edge $e\in\ECC(\GCC)$, the subcomplex 
$\DCC_e \setminus \left\{\Gamma_e\right\}$ is erasable in $\KG$.
\end{lem}
\begin{pf} 
Consider the discrete gradient specified in \cref{fig:cellAttach} as a gradient $\VSS_e$ on $\DCC_e \subseteq \KG$.
First note that $\DCC_e \setminus \left\{\Gamma_e\right\}$ is erasable in $\DCC_e$ through the gradient $\VSS_e \setminus \{(\omega_e,\Gamma_e)\}$.
Moreover, all 1-simplices of $\DCC_e$ that are paired in $\VSS_e$ with a 2-simplex do not appear in $\DCC_c$ for any edge $c \neq e$.
It follows that $\DCC_e \setminus \left\{\Gamma_e\right\}$ is erasable in  $\KG$.
\end{pf}

\begin{lem}\label{lem:feedbackhelper} Let $\CCC$ be a set of $2$-simplices such that $\KG \setminus \CCC$ is erasable, and let $\FCC = \{ f\in\ECC \mid \allowbreak \CCC\cap\DCC_f \neq \emptyset \}$. Then $\FCC$ is a feedback arc set of $\GCC$.
\end{lem}
\begin{pf} Each $2$-simplex $\sigma \in \CCC$ lies in $\DCC_f$ for a unique $f\in\ECC$, which implies $f\in\FCC$. 
In particular, $\sigma \in \CCC$ implies $\sigma \not \in \DCC_e$ for any $e \in \ECC \setminus \FCC$. 
Now consider the subgraph $\HCC = (\VCC, \ECC \setminus \FCC)$ of $\GCC$.
Then ${\KGH} \subseteq \KG \setminus \CCC$ is erasable, since any subcomplex of an erasable complex is easily seen to be erasable. 
Hence, by Lemma~\ref{lem:acycliceqerase}, $\HCC$ is acyclic, i.e., $\FCC$ is a feedback arc set of $\GCC$.
\end{pf}

\begin{prop} \label{prop:erkprop} Given an oriented degree-3 graph $\GCC$ and the corresponding complex $\KG$, $\erkg = \opt\minFAS(\GCC)$.
\end{prop}
\begin{pf} Given a graph $\GCC = (\VCC,\ECC)$, let $\FCC \subseteq \ECC$ be a minimum feedback arc set of $\GCC$, and let $\HCC=(\VCC,\ECC \setminus \FCC)$ be the corresponding maximum acyclic subgraph. We construct a new complex $\KCC^{\prime}$ from $\KG$ as follows: For every $f \in \FCC$, we remove $\{\Gamma_f\}$ from $\DCC_f \subseteq \KG$ and erase $\DCC_f \setminus \{\Gamma_f\}$ in $\KG$ using Lemma~\ref{lem:remainerase}. 
Note that $(\KG \setminus \{\Gamma_f \mid f \in \FCC\}) \searrow \KCC^{\prime}$.
In order to show that $\KCC^{\prime}$ is erasable, it suffices to show erasability of the pure subcomplex of~$\KCC^{\prime}$ induced by its $2$-simplices. It is easy to check that the subcomplex of $\KCC^{\prime}$ induced by the $2$-simplices in $\KCC^{\prime}$ is precisely $\KGH$. However, from Lemma~\ref{lem:acycliceqerase}, we can deduce that $\KGH$ is erasable. This implies that $\KCC^{\prime}$ is erasable, and hence $\left(\KG \setminus \left\{\Gamma_f \mid f \in \FCC\right\}\right)$ is erasable. Since the total number of 2-simplices that were removed to erase $\KG$ is equal to $\left| \FCC \right| = \opt\minFAS(\GCC)$, we have established that $\erkg \leq \opt\minFAS(\GCC)$.

Now assume for a contradiction that $\erkg < \opt\minFAS(\GCC)$. Let $\CCC$ be a minimal set of $2$-simplices that need to be removed to erase $\KG$, i.e., $\left|\CCC\right| = \erkg$. Let $\FCC^{\prime} = \{ f \in \ECC \mid \CCC\cap\DCC_f \neq \emptyset  \}$. By Lemma~\ref{lem:feedbackhelper}, the graph $(\VCC,\ECC \setminus\FCC^{\prime})$ is acyclic and $\FCC^{\prime}$ is a feedback arc set. Since each $2$-simplex lies in $\DCC_e$ for some unique $e\in\ECC$, it follows that 
$\left|\FCC^{\prime}\right| \leq \left|\CCC\right|$.
We conclude that
$\left|\FCC^{\prime}\right| < \opt\minFAS(\GCC)$, which contradicts the minimality of $\opt\minFAS(\GCC)$. Hence, the claim follows.
\end{pf}

In order to relate the homotopy type of $\GCC$ with that of $\KG$, we construct a new complex $\tildeKG$ as follows:
\begin{enumerate}
\item Start with a disjoint union of copies of $\DCC$, one for each edge in $\GCC$, denoted by $\DCC_e$.
\item Similar to the construction of $\KG$, the complex $\tildeKG$ is constructed by identifying some of the distinguished vertices of each gadget $\DCC_e$ based on the following rules, as applied to each vertex of $\GCC$ based on its indegree and outdegree:
\begin{enumerate}
	\item For every vertex of $\GCC$ that has incoming as well as outgoing edges, identify $s_j$ with $t_i$ for every incoming edge $i$  and outgoing edge $j$ of $\GCC$.
	\item For every vertex of $\GCC$ that has only incoming edges, identify all 0-simplices $t_e$ for every incoming edge $e$.
	\item For ever vertex of $\GCC$ that has only outgoing edges, identify all 0-simplices $s_e$ for every outgoing edge $e$.
\end{enumerate}
\end{enumerate}

\begin{lem}\label{lem:topoprop} 
$\tildeKG$ is homotopy equivalent to $\KG$.
\end{lem}
\begin{pf}
Comparing the two constructions, first note that $\KG$ can be obtained from $\tildeKG$ by further identifying certain 1-simplices $\phi_e \sim \omega_f$ or $\psi_e \sim \omega_f$ (together with vertices $v_e \sim u_f$ or $w_e \sim u_f$) in subcomplexes $\DCC_e,\DCC_f \subseteq \tildeKG$, where these two 1-simplices already have a common vertex by the identification $t_e \sim s_f$ in the construction of $\tildeKG$, and are otherwise not connected by another 1-simplex. 
In both $\tildeKG$ and $\KG$, the union of the two 1-simplices is contractible, and so each complex is homotopy equivalent to the space that further identifies each such pair of 1-simplices to a single point \cite[Proposition 0.17]{Hatcher2002Algebraic}. The claim follows.
\end{pf}

\begin{lem}\label{lem:tildeGcollapses}  $\tildeKG$ collapses to the undirected graph underlying $\GCC$. 
\end{lem}
\begin{pf} First note that $\omega_e$ is free in $\tildeKG$ for each $e$ in $\GCC$. Moreover, the only simplices that are possibly shared by gadgets $\DCC_i$ and $\DCC_j$ for $i \neq j$ are the vertices $s$ and $t$ of $\DCC_i$ and $\DCC_j$. Therefore, we can use the gradient vector field depicted in Figure~\ref{fig:cellAttach} to collapse each gadget $\DCC_e$ to the 1-simplex $\eta_e$ together with the vertices $s_e$ and $t_e$. 
Thus, $\tildeKG$ collapses to the subcomplex  $\QCC$ of $\tildeKG$ induced by these 1-simplices.
By construction of $\tildeKG$, this complex $\QCC$ is isomorphic to the undirected graph underlying $\GCC$.
\end{pf}

\begin{cor}\label{cor:homotopyeq} 
$\KG$ is homotopy equivalent to the undirected graph underlying $\GCC$. 
\end{cor}

\subsection{Inapproximability results}

Given an oriented connected degree-3 graph $\GCC=(\VCC,\ECC)$ and the corresponding complex $\KG$, let $\OPTA$ denote the optimal value of the \threeOMAS problem on $\GCC$, and let $\OPTB$ denote the optimal value of the \MaxMM problem on $\KG$.

We now describe an L-reduction from \threeOMAS to \MaxMM.
The map $\KCC: \GCC \mapsto \KCC(\GCC)$ transforms instances of
\threeOMAS (directed graphs) to instances of \MaxMM (simplicial complexes). 
The map $\ACC$
that transforms solutions of \MaxMM (discrete gradients $\VSS$ on $\KG$) to solutions of \threeOMAS (acyclic subgraphs $\ACC(\GCC,\VSS)$) is defined as follows:
Let \[\ensuremath{\FCC=\left\{ f\in\ECC \mid \exists\sigma\in\DCC_{f} : \sigma\textnormal{ is a critical 2-simplex in }\VSS\right\} }.\]
By Lemma~\ref{lem:feedbackhelper}, $\FCC$ is a feedback arc set.
The corresponding solution $\ACC(\GCC,\VSS)$ for \threeOMAS is then simply the subgraph of $\GCC$ with edges $\ECC \setminus \FCC$. 
The value of the objective function $\ALGB$ is the number of regular simplices in $\VSS$; the value of the objective function $\ALGA$  is the number of edges of the acyclic subgraph, $\left|\ECC \setminus \FCC\right|$.

For a discrete gradient $\VSS$ on $\KG$, let $n$ denote the number of simplices in $\KG$, let $m$ denote the total number of critical simplices in $\VSS$, and let $m_d$ denote the number of critical simplices in dimension $d$. Also, let $\beta_d$ denote the Betti number of $\KG$ in dimension $d$, and let $\beta$ be the sum of all Betti numbers. 

\begin{lem} \label{lem:mbequalities} With the above notation,
\[\ALGB \leq n - 2\, m_2 - \beta \quad \text{and} \quad \ALGA \geq |\ECC| - m_2.\]
\end{lem} 
\begin{pf}
 By the Morse inequalities~\cite[Theorem~3.7]{Fo98}, we have $m_0 \geq \beta_0$ and
\[ m_2 - m_1 + m_0 = \beta_2 - \beta_1 + \beta_0 . \]
From \cref{cor:homotopyeq}, we have $\beta_0=1$ and $\beta_2= 0$.
This gives us:
\begin{equation} \label{eq:morserel}
m_1 = \beta_1 + m_2 + (m_0 - \beta_0).
\end{equation}
Moreover, 
\begin{align*} 
m & = m_2 + m_1 + m_0 && \textnormal{(by definition)}\\
  & =  2\, m_2 + \beta_1 + 2\, m_0 -\beta_0 && \textnormal{(from Equation~\eqref{eq:morserel})}\\
  & = 2\, m_2 + \beta + 2\,(m_0 -\beta_0) && \textnormal{(since $\beta_2 = 0$)} \\
  & \geq 2\, m_2 + \beta  && \textnormal{(since $m_0 \geq \beta_0$).} 
\end{align*}
Hence, $\ALGB = n - m \leq n - 2\, m_2 - \beta$.

In the construction of the acyclic subgraph $\ACC(\GCC,\VSS)$, 
for every critical $2$-simplex in $\VSS$, we remove at most one edge in $\GCC$. Hence, we conclude that $\ALGA \geq |\ECC| - m_2$.
\end{pf}

\begin{lem}\label{lem:optlemma} Given a graph $\GCC$ and the corresponding complex $\KG$, 
\[\OPTB = n -  2\,\erkg - \beta = n - 2\,\opt\minFAS(\GCC) - \beta . \]
\end{lem}   
\begin{pf} 
First note that for an optimal gradient vector field on $\KG$, we have $m_0 = 1$ and $m_2 = \erkg$.
The first equality now follows by observing that in the proof of Lemma~\ref{lem:mbequalities}, equality $\ALGB = n - 2\, m_2 - \beta$ is obtained for $\beta_0 = m_0$. The second equality follows immediately follows from Proposition~\ref{prop:erkprop}.
\end{pf}

\begin{lem}\label{lem:resize} $\OPTB \leq 78 \, \OPTA$.
\end{lem}    
\begin{pf}
For the Max-Acyclic Subgraph problem, from the trivial $\frac{1}{2}$-factor approximation algorithm mentioned in~\cite[Ch.~1]{Vaz}, one knows that it is always possible to find an acyclic subgraph $\ACC_\HCC$ of a directed graph $\HCC$ that contains at least half the number of edges in $\HCC$. Clearly, this bound continues to hold when the class of graphs is restricted to degree-$3$ oriented graphs. This gives the following inequality:
\begin{equation} \label{eq:MASOPT}
\OPTA \geq \frac{\edgecardinality}{2} .
\end{equation} 

First note that the number of simplices in the modified dunce hat $\DCC$ is $7+19+13=39$. 
The complex $\KG$ described in \cref{sec:hardnessmax} is constructed from a disjoint union of $\edgecardinality$ copies of $\DCC$ with several simplices identified, giving us
\begin{equation} \label{eq:size}
n \leq 39\, \edgecardinality .
\end{equation}
From \cref{lem:optlemma,eq:size,eq:MASOPT} we obtain the bound
\[ \OPTB \leq n  \leq  39\, \edgecardinality   \leq 78 \, \OPTA .
\qedhere
\]
\end{pf}

\begin{lem}\label{lem:errorbound} 
\[\left( \OPTA - \ALGA \right) \leq \frac{1}{2}\left( \OPTB - \ALGB \right) . \]
\end{lem}
\begin{pf} By definition, $\OPTA = \ECC - \opt\minFAS(\GCC)$. By Lemma~\ref{lem:mbequalities}, \[\ALGA \geq |\ECC| - m_2.\] Hence,
\begin{equation}\label{eq:LHS}
\left( \OPTA - \ALGA \right) \leq m_2 - \opt\minFAS(\GCC) .
\end{equation}
Using Lemma~\ref{lem:mbequalities} and Lemma~\ref{lem:optlemma}, we obtain
\begin{equation}\label{eq:RHS}
\left( \OPTB - \ALGB \right) \geq 2\,(m_2 - \opt\minFAS(\GCC)) .
\end{equation}
Substituting Equation~\ref{eq:LHS} in Equation~\ref{eq:RHS}, we obtain the lemma.
\end{pf}

\begin{thm} It is NP-hard to approximate \MaxMM within a factor of $\left(1 - \frac{1}{4914}\right)+\epsilon$ and UGC-hard to approximate it within a factor of $\left(1 - \frac{1}{702}\right)+\epsilon$, for any $\epsilon>0$.
\end{thm}
\begin{pf} From Lemma~\ref{lem:resize} and Lemma~\ref{lem:errorbound} we conclude that the reduction from \threeOMAS to \MaxMM is an L-reduction with parameters $\mu=78$ and $\nu=\frac{1}{2}$. By Theorem~\ref{thm:newfactor}, if there exists a $\big( 1 - \frac{1}{\mu\nu} \delta + \epsilon \big)$-approximation algorithm for \MaxMM, then there exists a  $\left(1 - \delta + \mu\nu\epsilon\right)$-algorithm for \threeOMAS. Using Corollary~\ref{cor:hammer}, we choose $\delta = \frac{1}{126}$ to deduce that it is NP-hard to approximate \MaxMM within a factor of $
\left(1 - \frac{1}{4914}\right)+\epsilon$,
and choose $\delta = \frac{1}{18}$ to deduce that it is UGC-hard to approximate \MaxMM within a factor of $
\left(1 - \frac{1}{702}\right)+\epsilon$.
\end{pf}

\section{Conclusion \& Discussion} 

In this paper, we provide the first hardness of approximation results for
the maximization and the minimization variants of the Morse matching
problems. 

While we established a hardness result for Min-Morse Matching on simplicial complexes
of dimension $d\geq3$, the question
of hardness of approximation for Min-Morse matching for 2-dimensional
simplicial complexes remains open. We will address this question in future work.

For the Max-Morse Matching problem on $d$-dimensional simplicial complexes, although our work clears a major hurdle of going beyond NP-hardness, a gap remains between the best approximability and inapproximability bounds. The best known approximation algorithm for Max-morse matching on simplicial complexes yields an approximation ratio of  $\frac{d+1}{d^{2}+d+1}$~\cite{RBN17}. We believe that our result and techniques will pave way for further work in improving the gap, and in placing Max-Morse Matching in the right kind of approximation-algorithms related complexity class. In particular, devising an approximation algorithm for Max-Morse Matching with an approximation factor independent of the dimension of the complex, or establishing a hardness of approximation result for Max-Morse Matching that is dependent on the dimension is a challenging open problem.

 We close the discussion with some additional open problems. 
 Note that the complex $\KG$ employed in the hardness result for Max-Morse Matching described in \cref{sec:hardnessmax} is not a manifold. Hence, the question of hardness of approximation for Max-Morse Matching on simplicial manifolds is open. The best known approximation algorithm for Max-Morse Matching on $d$-dimensional simplicial manifolds has approximation factor $\frac{2}{d}$~\cite{RBN17}. 
 
 Finally, we also leave it as an open question to investigate sharper inapproximability bounds for Max-Morse Matching on regular cell complexes.

\bibliography{hard-morse.bib}

\end{document}